\documentclass[11pt]{amsart}
\newtheorem{theorem}{Theorem}[section]
\newtheorem{lemma}[theorem]{Lemma}
\newtheorem{proposition}[theorem]{Proposition}
\newtheorem{corollary}[theorem]{Corollary}

\theoremstyle{definition}

\newtheorem{example}[theorem]{Example}

\usepackage{amscd,amssymb}

\begin{document}


%
%

\title[Weitzenb\"ock derivations of free algebras]
{Weitzenb\"ock derivations\\
of free metabelian associative algebras}

\author[Dangovski, Drensky, F\i nd\i k]
{Rumen Dangovski, Vesselin Drensky, {\c S}ehmus Findik}
\address{Massachusetts Institute of Technology,\\
428 Memorial Drive, Cambridge, MA 02139, U.S.A.}
\email{rumenrd@mit.edu}

\address{Institute of Mathematics and Informatics,\\
Bulgarian Academy of Sciences,
1113 Sofia, Bulgaria}
\email{drensky@math.bas.bg}

\address{Department of Mathematics,\\
\c{C}ukurova University, 01330 Balcal\i,
 Adana, Turkey}
\email{sfindik@cu.edu.tr}

\maketitle

\begin{abstract} By the classical theorem of Weitzenb\"ock the algebra of constants $K[X_d]^{\delta}$
of a nonzero locally nilpotent linear derivation $\delta$ of the polynomial algebra
$K[X_d]=K[x_1,\ldots,x_d]$ in several variables over a field $K$ of characteristic 0
is finitely generated. As a noncommutative generalization one considers the algebra of constants
$F_d({\mathfrak V})^{\delta}$ of a locally nilpotent linear derivation $\delta$ of a finitely generated
relatively free algebra $F_d({\mathfrak V})$ in a variety $\mathfrak V$ of unitary associative algebras over $K$.
It is known that $F_d({\mathfrak V})^{\delta}$ is finitely generated if and only if $\mathfrak V$
satisfies a polynomial identity which does not hold for the algebra $U_2(K)$ of $2\times 2$ upper triangular matrices.
Hence the free metabelian associative algebra $F_d=F_d({\mathfrak M})=F_d({\mathfrak N}_2{\mathfrak A})=F_d(\text{var}(U_2(K)))$
is a crucial object to study.
We show that the vector space of the constants $(F_d')^{\delta}$ in the commutator ideal $F_d'$
is a finitely generated $K[U_d,V_d]^{\delta}$-module, where $\delta$ acts on $U_d$ and $V_d$ in the same way as on $X_d$.
For small $d$, we calculate the Hilbert series of $(F_d')^{\delta}$
and find the generators of the $K[U_d,V_d]^{\delta}$-module $(F_d')^{\delta}$.
This gives also an (infinite) set of generators of the algebra
$F_d^{\delta}$.
\end{abstract}



\section{Introduction}

Recall that a derivation of an algebra $R$ over a field $K$ is a linear operator $\delta:R\to R$ such that
$\delta(uv)=\delta(u)v+u\delta(v)$ for all $u,v\in R$. In this paper we fix a field $K$
of characteristic 0, an integer $d\geq 2$ and a set of variables
$X_d=\{x_1,\ldots,x_d\}$. For the polynomial algebra $K[X_d]=K[x_1,\ldots,x_d]$
in $d$ variables, every mapping $\delta: X_d\to K[X_d]$ can be extended
to a derivation of $K[X_d]$ which we shall denote by the same symbol $\delta$.
We assume that $\delta$ is a Weitzenb\"ock derivation, i.e., acts as a nonzero nilpotent linear operator of the vector space
$KX_d$ with basis $X_d$. Up to a change of the basis of $KX_d$, the derivation $\delta$ is determined by its
Jordan normal form $J(\delta)$
with Jordan cells $J_1,\ldots,J_s$ with zero diagonals.
Hence the essentially different Weitzenb\"ock derivations
are in a one-to-one correspondence with the partition
$(p_1+1,\ldots,p_s+1)$ of $d$, where $p_1\geq \cdots\geq p_s\geq 0$,
$(p_1+1)+\cdots+(p_s+1)=d$, and the correspondence is given
in terms of the size $(p_i+1)\times (p_i+1)$ of the Jordan cells $J_i$ of $J(\delta)$, $i=1,\ldots,s$.
We shall denote the derivation corresponding to this partition by $\delta(p_1,\ldots,p_s)$.
The algebra of constants of $\delta$
\[
K[X_d]^{\delta}=\ker{\delta}=\{u\in K[X_d]\mid \delta(u)=0\}
\]
can be studied also with methods of classical invariant theory because it coincides with
the algebra of invariants $K[X_d]^{UT_2(K)}$ of the unitriangular group
$UT_2(K)\cong\{\exp(\alpha\delta)\mid\alpha\in K\}$.
The classical theorem of Weitzenb\"ock \cite{W} states that for any Weitzenb\"ock derivation $\delta$
the algebra of constants $K[X_d]^{\delta}$ is finitely generated, see the book by Nowicki \cite{N}
for more details.

The polynomial algebra $K[X_d]$ is free in the class of all commutative algebras.
As a noncommutative generalization one considers
the relatively free algebra $F_d({\mathfrak V})$ in a variety $\mathfrak V$ of unitary associative algebras over $K$,
see e.g., \cite{D1} for a background on varieties of algebras.
As in the polynomial case, $F_d({\mathfrak V})$ is freely generated by the set $X_d$ and every map
$X_d\to F_d({\mathfrak V})$ can be extended to a derivation of $F_d({\mathfrak V})$. Again, we shall call
the derivations $\delta$ which act as nilpotent linear operators of the vector space
$KX_d$ Weitzenb\"ock derivations and shall denote them by $\delta(p_1,\ldots,p_s)$.
In the theory of varieties of unitary associative algebras
there is a dichotomy. Either all polynomial identities of the variety  $\mathfrak V$ follow from the metabelian identity $[x_1,x_2][x_3,x_4]=0$
(which is equivalent to the condition that $\mathfrak V$ contains the algebra $U_2(K)$
of $2\times 2$ upper triangular matrices), or $\mathfrak V$
satisfies an Engel polynomial identity $[x_2,\underbrace{x_1,\ldots,x_1}_{n\text{ times}}]=x_2\text{ad}^nx_1=0$.
Drensky and Gupta \cite{DG} studied Weitzenb\"ock derivations $\delta$ acting on $F_d({\mathfrak V})$.
In particular, if $U_2(K)\in\mathfrak V$, then the algebra of constants $F_d({\mathfrak V})^{\delta}$ is not finitely generated.
If $U_2(K)$ does not belong to $\mathfrak V$, a result of Drensky \cite{D2}
gives that the algebra $F_d({\mathfrak V})^{\delta}$ is finitely generated.
Hence the free metabelian associative algebra
\[
F_d=F_d({\mathfrak M})=F_d({\mathfrak N}_2{\mathfrak A})=F_d(\text{var}(U_2(K)))
\]
is a crucial object in the study of Weitzenb\"ock derivations.

The present paper may be considered as a continuation of our recent paper \cite{DDF}
where we have studied the algebra of constants $(L_d/L_d'')^{\delta}$ of the Weitzenb\"ock derivation $\delta$ acting on the
free metabelian Lie algebra $L_d/L_d''$, where $L_d$ is the free $d$-generated Lie algebra. As in the Lie case we show
that the algebra of constants $F_d^{\delta}$ when $\delta$ acts on the free metabelian associative algebra
$F_d=F_d(\text{var}(U_2(K)))$ is very close to finitely generated. In the Lie case the commutator ideal $L_d'/L_d''$ has
a natural structure of a $K[X_d]$-module. In \cite{DDF} we have seen that the vector space of constants $(L_d'/L_d'')^{\delta}$
is a finitely generated $K[X_d]^{\delta}$-module. In the associative case the commutator ideal $F_d'$ is a $K[X_d]$-bimodule.
We prove that the vector space of constants $(F_d')^{\delta}$ is a finitely generated
$K[U_d,V_d]^{\delta}$-module, where $\delta$ acts on the variables $U_d$ and $V_d$ in the same way as on the variables $X_d$.
Then, using methods of \cite{BBD} we give an algorithm how to calculate the Hilbert series of $F_d^{\delta}$
and calculate it for small $d$.
If the Jordan form of $\delta$ contains a $1\times 1$ cell, we may assume that $\delta$ acts as a nilpotent linear operator
on $KX_{d-1}$ and $\delta(x_d)=0$. In this situation we express the Hilbert (or Poincar\'e) series of
$F_d^{\delta}$ and the generators of the $K[U_d,V_d]^{\delta}$-module $(F_d')^{\delta}$ in terms of the Hilbert series of
$F_{d-1}^{\delta}$ and the generators of the $K[U_{d-1},V_{d-1}]^{\delta}$-module $(F_{d-1}')^{\delta}$.
Finally, we find the generators of the $K[U_d,V_d]^{\delta}$-module $(F_d')^{\delta}$  for $d\leq 6$
(and $\delta\not=\delta(5),\delta(3,1)$), which
gives also an explicit (infinite) set of generators of the algebra $F_d^{\delta}$.

\section{Finite generation}

Let $A_d=K\langle X_d\rangle$ be the free unitary associative
algebra of rank $d$ with free generating set $X_d=\{x_1,\ldots,x_d
\}$, and let $A_d'$ be its commutator ideal, i.e., the ideal
generated by all commutators
\[
[f,g]=f\text{ad}g=fg-gf,\quad f,g\in A_d.
\]
The metabelian variety $\mathfrak M$ consists of all associative algebras satisfying the polynomial identity $[x_1,x_2][x_3,x_4]=0$.
The free metabelian algebra $F_d=F_d(\mathfrak M)$ of rank $d$ is isomorphic to the factor-algebra $A_d/(A'_d)^2$.
Clearly, $K[X_d]\cong A_d/A_d'\cong F_d/F_d'$.
We assume that all Lie commutators are left normed and
\[
[x_1,\ldots,x_{n-1},x_n]=[[x_1,,\ldots,x_{n-1}],x_n]=[x_1,\ldots,x_{n-1}]\text{ad}x_n.
\]
It is well known, see \cite{D1}, that the commutator ideal $F_d'$ of $F_d$ has a basis
consisting of all
\[
x_1^{a_1}\cdots x_d^{a_d}[x_{j_1},x_{j_2},x_{j_3},\ldots,x_{j_n}],\quad a_i\geq 0, 1\leq j_i\leq d, j_1>j_2\leq j_3\leq\cdots\leq j_n.
\]
The metabelian identity implies the identity
\[
x_{i_{\rho(1)}}\cdots x_{i_{\rho(m)}}[x_{j_1},x_{j_2},x_{j_{\sigma(3)}},\ldots,x_{j_{\sigma(n)}}]
=x_{i_1}\cdots x_{i_m}[x_{j_1},x_{j_2},x_{j_3},\ldots,x_{j_n}],
\]
where $\rho$ and $\sigma$ are arbitrary permutations of $1,\ldots,m$ and $3,\ldots,n$, respectively. This identity allows to define
an action of the polynomial algebra $K[U_d,V_d]$ on $F_d'$, where
$U_d=\{u_1,\ldots,u_d\}$ and $V_d=\{v_1,\ldots,v_d\}$ are two sets of commuting variables:
If $f\in F_d'$, then
\[
fu_i=x_if, \quad fv_i=f\text{\rm ad}x_i, \quad i=1,\ldots,d.
\]
In this way, the vector subspace $F_d'$ of $F_d$ is a $K[U_d,V_d]$-module (or a $K[X_d]$-bimodule).

Now we need an embedding of $F_d$ into a wreath product. The construction is a partial case
of the construction of Lewin \cite{L} used in \cite{DG} and is similar to the construction of
Shmel'kin \cite{Sh} in the case of free metabelian Lie algebras (as used in \cite{DDF}).
Let $Y_d = \{y_1, \ldots, y_d\}$ be a set of commuting variables and let
\[
M_d=\bigoplus_{i=1}^da_iK[U_d,V_d']=A_dK[U_d,V_d']
\]
be the free $K[U_d,V_d']$-module of rank $d$ freely generated by $A_d=\{a_1,\ldots,a_d\}$, where $V_d'=\{v_1',\ldots,v_d'\}$.
We equip $M_d$ with trivial multiplication
$M_d\cdot M_d=0$ and with a structure of a free $K[Y_d]$-bimodule with the action of $K[Y_d]$ defined by
\[
y_ja_i = a_iu_j, \quad a_iy_j = a_iv_j', \quad i,j = 1, \ldots, d.
\]
Then the algebra $W_d=K[Y_d]\rightthreetimes M_d$ satisfies the metabelian identity and hence belongs to $\mathfrak M$.
The following proposition is a partial case of \cite{L}.

\begin{proposition}\label{embedding}
The mapping $\imath :x_j\to y_j+a_j$, $j=1,\ldots,d$, defines an embedding $\imath$ of $F_d$ into $W_d$.
\end{proposition}

If
\[
w=\sum_{i>j}\sum_{k,l}\beta_{ijkl}X_d^k[x_i,x_j]\text{\rm ad}^lX_d,
\]
\[
\beta_{ijkl}\in K,\quad X_d^k=x_1^{k_1}\cdots x_d^{k_d},\quad \text{\rm ad}^lX_d=(\text{ad}^{l_1}x_1)\cdots(\text{ad}^{l_d}x_d),
\]
then
\[
\imath(w)=\sum_{i>j}\sum_{k,l}\beta_{ijkl}(a_iv_j-a_jv_i)U_d^kV_d^l,\quad U_d^k=u_1^{k_1}\cdots u_d^{k_d},\quad
V_d^l=v_1^{l_1}\cdots v_d^{k_l},
\]
and $V_d=\{v_1,\ldots,v_d\}$, $v_i=v_i'-u_i$, $i=1,\ldots,d$.
We may replace the variables $V_d'$ with $V_d$. In this way $M_d$ becomes a free $K[U_d,V_d]$-module.
To simplify the notation we shall omit $\imath$ and shall consider
$F_d$ as a subalgebra of $W_d$. Since the action of $K[U_d,V_d]$ on
$F_d'$ agrees with its action on $M_d$, we shall also think that
$F_d'$ is a $K[U_d,V_d]$-submodule of $M_d$.
If $\delta$ is a Weitzenb\"ock derivation of $F_d$, such that
\[
\delta(x_j)=\sum_{i=1}^d\alpha_{ij}x_i,\quad \alpha_{ij}\in K,\quad i,j=1,\ldots,d,
\]
we define an action of $\delta$ on $W_d$ assuming that
\[
\delta(a_j)=\sum_{i=1}^d\alpha_{ij}a_i,\quad \delta(y_j)=\sum_{i=1}^d\alpha_{ij}y_i,
\]
\[
\delta(u_j)=\sum_{i=1}^d\alpha_{ij}u_i,\quad \delta(v_j)=\sum_{i=1}^d\alpha_{ij}v_i,
\quad j=1,\ldots,d.
\]
Obviously, if $w\in F_d$, then $\imath(\delta(w))=\delta(\imath(w))$.
It is clear that the vector space $M_d^{\delta}$ of the constants of $\delta$ in the $K[U_d,V_d]$-module $M_d$
is a $K[U_d,V_d]^{\delta}$-module. The following lemma is a partial case of \cite[Proposition 3]{D2}.

\begin{lemma}\label{finite generation of module}
The vector space $M_d^{\delta}$ is a finitely generated $K[U_d,V_d]^{\delta}$-module.
\end{lemma}

The next theorem is a direct consequence of the lemma.

\begin{theorem}\label{theorem for finite generation}
Let $\delta$ be a Weitzenb\"ock derivation of the free metabelian associative algebra $F_d=F_d({\mathfrak M})$. Then
the vector space $(F_d')^{\delta}$ of the constants of $\delta$ in the commutator
ideal $F_d'$ of $F_d$ is a finitely generated $K[U_d,V_d]^{\delta}$-module.
\end{theorem}

\begin{proof}
By Lemma \ref{finite generation of module} the $K[U_d,V_d]^{\delta}$-module $M_d^{\delta}$ is finitely generated.
By the theorem of Weitzenb\"ock \cite{W} the algebra $K[U_d,V_d]^{\delta}$ is also finitely generated. Hence
all $K[U_d,V_d]^{\delta}$-submodules of $M_d^{\delta}$, including $(F_d')^{\delta}$, are also finitely generated.
\end{proof}

\section{Hilbert series}\label{section Hilbert series}

The free metabelian associative algebra $F_d=F_d({\mathfrak M})$
is a graded vector space. If $F_d^{(n)}$ is the homogeneous component of degree $n$
of $F_d$, then the Hilbert (or Poincar\'e) series of $F_d$ is the formal power series
\[
H(F_d,z)=\sum_{n\geq 0}\dim F_d^{(n)}z^n.
\]
The algebra $F_d$ is also multigraded, with a ${\mathbb Z}^d$-grading which counts the degree of each variable $x_j$ in the monomials
in $F_d$. If $F_d^{(n_1,\ldots,n_d)}$ is the multihomogeneous component of degree $(n_1,\ldots,n_d)$, then the corresponding Hilbert series
of $F_d$ is given by
\[
H(F_d,z_1,\ldots,z_d)=\sum_{n_j\geq 0}\dim F_d^{(n_1,\ldots,n_d)}z_1^{n_1}\cdots z_d^{n_d}.
\]
If $\delta$ is a Weitzenb\"ock derivation of $F_d$, then the algebra of constants $F_d^{\delta}$ is also graded and its Hilbert series is
\[
H(F_d^{\delta},z)=\sum_{n\geq 0}\dim (F_d^{\delta})^{(n)}z^n.
\]

The Hilbert series of $F_d({\mathfrak M})$ is well known. We shall give the proof for self-containedness.

\begin{lemma}\label{Hilbert series of free metabelian associative algebra}
The Hilbert series of the free metabelian associative algebra $F_d$ is
\[
H(F_d,z_1,\ldots,z_d)=\prod_{j=1}^d\frac{1}{1-z_j}\left(2+(z_1+\cdots+z_d-1)\prod_{j=1}^d\frac{1}{1-z_j}\right).
\]
\end{lemma}

\begin{proof}
The Hilbert series of a unitary relatively free algebra $F_d({\mathfrak V})$ can be expressed in terms of the Hilbert series of the
subalgebra $B_d({\mathfrak V})$ of the so called proper (or commutator) polynomials, see, e.g. \cite{D1}:
\[
H(F_d({\mathfrak V}),z_1,\ldots,z_d)=H(K[X_d],z_1,\ldots,z_d)H(B_d({\mathfrak V}),z_1,\ldots,z_d)
\]
\[
=\prod_{i=1}^d\frac{1}{1-z_i}H(B_d({\mathfrak V}),z_1,\ldots,z_d).
\]
For the metabelian variety $\mathfrak M$ the algebra of proper polynomials is spanned by the base field $K$
and the commutator ideal $L_d'/L_d''$ of the free metabelian Lie algebra.
The Hilbert series of $L_d'/L_d''$ is
\[
H(L_d'/L_d'',z_1,\ldots,z_d)=1+(z_1+\cdots+z_d-1)\prod_{i=1}^d\frac{1}{1-z_i}.
\]
Hence
\[
H(F_d({\mathfrak M}),z_1,\ldots,z_d)=\prod_{i=1}^d\frac{1}{1-z_i}(1+H(L_d'/L_d'',z_1,\ldots,z_d))
\]
which gives immediately the result.

The Hilbert series of $F_d({\mathfrak M})$ can be calculated directly using that $F_d({\mathfrak M})$
has a vector space basis consisting of
\[
X_d^a,\quad X_d^a[x_{j_1},x_{j_2},x_{j_3},\ldots,x_{j_n}],\quad X_d^a=x_1^{a_1}\cdots x_d^{a_d},
j_1>j_2\leq j_3\leq\cdots\leq j_n,n\geq 2.
\]
The products $X_d^a$ contribute to the factor $\displaystyle \prod_{i=1}^d\frac{1}{1-z_i}$. The expressions
$[x_{j_1},x_{j_2},x_{j_3},\ldots,x_{j_n}]$ for $n\geq 1$ and without the restriction $j_1>j_2$
behave as the pairs $(x_{j_1},x_{j_2}x_{j_3}\cdots x_{j_n})\in X_d\times [X_d]$, where $[X_d]$ is the semigroup
of all monomials in $K[X_d]$. Hence they give the series
\[
(z_1+\dots+z_d)\prod_{i=1}^d\frac{1}{1-z_i}.
\]
Now we have to subtract the series $\displaystyle \prod_{i=1}^d\frac{1}{1-z_i}-1$ which corresponds to the expressions
$[x_{j_1},x_{j_2},x_{j_3},\ldots,x_{j_n}]$ for $n\geq 1$ and $j_1\leq j_2$ (and behave as the monomials
$x_{j_1}x_{j_2}x_{j_3}\cdots x_{j_n}\in [X_d]$).
\end{proof}

Starting from the Hilbert series of the algebra $F_d=F_d({\mathfrak M})$ and the sizes of the Jordan cells of
the Weitzenb\"ock derivation $\delta$, there are many ways to compute the Hilbert series
of the algebra of constants $F_d^{\delta}$, see the comments in \cite{BBD} and \cite{DDF}.
We shall sketch the way used in \cite{DDF} in the Lie case.
It is an improvement (given in \cite{BBD}) of the method
of Elliott \cite{E} and its further development by McMahon \cite{Mc},
the so called partition analysis or $\Omega$-calculus.

Let $\delta=\delta(p_1,\ldots,p_s)$ be a Weitzenb\"ock derivation of $F_d$.
As we already mentioned, the algebra of constants $F_d^{\delta}$ coincides with the algebra of $UT_2(K)$-invariants
$F_d^{UT_2(K)}$, where $UT_2(K)\cong\{\exp(\alpha\delta)\mid\alpha\in K\}$. The idea is to extend the action of
$UT_2(K)$ on $F_d$ to an action of the general linear group $GL_2(K)$. Then $F_d$ is a direct sum of irreducible $GL_2(K)$-modules.
Each irreducible component contains a one-dimensional subspace of $UT_2(K)$-invariants
and the algebra $F_d^{\delta}=F_d^{UT_2(K)}$ is spanned by these subspaces.
In order to compute the Hilbert series of $F_d^{\delta}$ it is sufficient to find the number of the irreducible $GL_2(K)$-components
in each homogeneous subspace $F_d^{(n)}$ of $F_d$. The Hilbert series $H(F_d,z_1,\ldots,z_d)$ plays the role of the character
of $GL_d(K)$ in the following sense.
If $g$ belongs to the diagonal subgroup $D_d(K)$ of $GL_d(K)$ and has eigenvalues $\zeta_1,\ldots,\zeta_d$, then
\[
H(F_d,\zeta_1z,\ldots,\zeta_dz)=\sum_{n\geq 0}\text{tr}_{F_d^{(n)}}(g)z^n.
\]
The irreducible polynomial $GL_2(K)$-modules are indexed by partitions $\lambda=(\lambda_1,\lambda_2)$, $\lambda_1\geq\lambda_2$.
If $W=W(\lambda)$ is an irreducible component of $F_d$, then it is a direct sum of one-dimensional subspaces $W^{\alpha}=W^{(\alpha_1,\alpha_2)}$
such that $\alpha_1+\alpha_2=\lambda_1+\lambda_2$ and $\lambda_1\geq\alpha_1,\alpha_2\geq\lambda_2$. If $g\in D_2(K)$ has eigenvalues
$\xi_1,\xi_2$ and $w\in W^{\alpha}$, then $g(w)=\xi_1^{\alpha_1}\xi_2^{\alpha_2}w$. The unique
(up to a multiplicative constant) nonzero element $w\in W^{\lambda}$ spans the $UT_2(K)$-invariant subspace of $W$.
In our case $GL_2(K)$ acts polynomially on the vector space $KX_d$ and the vector space $KX_d$ has a basis
with the property that the diagonal subgroup $D_2(K)$ of $GL_2(K)$ acts as a group of $d\times d$ diagonal matrices.
Then the $GL_2(K)$-character of $F_d$ is a formal power series
\[
\chi_{F_d}(g,z)=\sum_{n\geq 0}\chi_{F_d^{(n)}}(g)z^n=\sum_{n\geq 0}\text{tr}_{F_d^{(n)}}(g)z^n,
\]
\[
g=\left(\begin{matrix}\xi_1&0\\
0&\xi_2\\
\end{matrix}\right)=\xi_1e_{11}+\xi_2e_{22}\in D_2(K)\subset GL_2(K).
\]
The characters of the irreducible $GL_2(K)$-modules $W(\lambda)$ are equal to the Schur functions $S_{\lambda}(t_1,t_2)$ and
for $g=\xi_1e_{11}+\xi_2e_{22}\in D_2(K)\subset GL_2(K)$
\[
\chi_{F_d}(g,z)=\sum_{n\geq 0}\sum_{(\lambda_1,\lambda_2)}m(\lambda_1,\lambda_2,n)S_{(\lambda_1,\lambda_2)}(\xi_1,\xi_2)z^n.
\]
Here the nonnegative integer $m(\lambda_1,\lambda_2,n)$ is the multiplicity of $W(\lambda_1,\lambda_2)$ in $F_d^{(n)}$.
Then the problem of computing of the Hilbert series of $F_d^{\delta}$ can be solved in two steps. First we have to find the Hilbert series
\[
H_{GL_2}(F_d,t_1,t_2,z)=\sum_{n\geq 0}\sum_{(\lambda_1,\lambda_2)}m(\lambda_1,\lambda_2,n)S_{(\lambda_1,\lambda_2)}(t_1,t_2)z^n
\]
which corresponds to the $GL_2(K)$-action on $F_d$. Then
\[
H(F_d^{\delta},z)=\sum_{n\geq 0}\sum_{(\lambda_1,\lambda_2)}m(\lambda_1,\lambda_2,n)z^n
\]
and we have to use $H_{GL_2}(F_d,t_1,t_2,z)$ and to evaluate the series $H(F_d^{\delta},z)$.

The first part of the problem is solved in the following way.
We assume that $X_d$ is a Jordan basis for the action of $\delta$ on the vector space $KX_d$.
If $Y_i=\{x_j,x_{j+1},\ldots,x_{j+p_i}\}$ is the part of the basis $X_d$ corresponding to the $i$-th Jordan cell of $\delta$,
we have an action of the linear automorphism $\exp(\delta)$ on the vector space $KY_i$ which is the same as its action on
the vector space of the binary forms of degree $p_i$. We extend this action to the action of $GL_2(K)$. In this way
$KX_d$ is a direct sum of the $GL_2(K)$-modules of the binary forms of degree $p_i$, $i=1,\ldots,s$. Then we extend this action
of $GL_2(K)$ diagonally on the whole $F_d$. The basis $\displaystyle X_d=\bigcup_{i=1}^sY_i$ consists of eigenvectors
of the diagonal subgroup $D_2(K)$ of $GL_2(K)$.
If $g=\xi_1e_{11}+\xi_2e_{22}\in D_2(K)$, then
\[
g(x_{j+k})=\xi_1^{p_i-k}\xi_2^kx_{j+k}, \quad k=0,1,\ldots,p_i.
\]
This defines a bigrading on $F_d$ assuming that the bidegree of $x_{j+k}$ is $(p_i-k,k)$.
We express the Hilbert series of $F_d$ as a bigraded vector space. For this purpose we replace in the Hilbert series
$H(F_d,z_1,\ldots,z_d)$ the variables $z_j,z_{j+1},\ldots,z_{j+p_i-1},z_{j+p_i}$ corresponding
to each set $Y_i=\{x_j,x_{j+1},\ldots,x_{j+p_i-1},x_{j+p_i}\}$
by $t_1^{p_i}z,t_1^{p_i-1}t_2z,\ldots,t_1t_2^{p_i-1}z,t_2^{p_i}z$, respectively, and obtain the Hilbert series
\[
H_{GL_2}(F_d,t_1,t_2,z)=H(F_d,t_1^{p_1}z,t_1^{p_1-1}t_2z,\ldots,t_2^{p_1}z,\ldots,t_1^{p_s}z,t_1^{p_s-1}t_2z,\ldots,t_2^{p_s}z)
\]
\[
=\sum_{n\geq 0}\sum_{(\lambda_1,\lambda_2)}m(\lambda_1,\lambda_2,n)S_{(\lambda_1,\lambda_2)}(t_1,t_2)z^n.
\]
The variable $z$ gives the total degree and $t_1,t_2$ count the bidegree
induced by the action of the diagonal subgroup of $GL_2(K)$.
The coefficient of $t_1^{n_1}t_2^{n_2}z^n$ in $H_{GL_2}(F_d,t_1,t_2,z)$ is equal to the dimension of the elements of $F_d$ which are
linear combinations of products of length $n$ in the variables $X_d$ and are of bidegree $(n_1,n_2)$.
Then the bigraded Hilbert series of the algebra $F_d^{\delta}$ is
\[
H_{GL_2}(F_d^{\delta},t_1,t_2,z)=\sum_{n\geq 0}\sum_{(\lambda_1,\lambda_2)}m(\lambda_1,\lambda_2,n)t_1^{\lambda_1}t_2^{\lambda_2}z^n
\]
and the Hilbert series of $F_d^{\delta}$ as a $\mathbb Z$-graded vector space is
\[
H(F_d^{\delta},z)=\sum_{n\geq 0}\dim(F_d^{\delta})^{(n)}z^n=H_{GL_2}(F_d^{\delta},1,1,z).
\]
In this way the second part of the problem reduces to the computing of the series $H_{GL_2}(F_d^{\delta},t_1,t_2,z)$.
The Hilbert series $H_{GL_2}(F_d,t_1,t_2,z)$ is a symmetric function with respect to the variables $t_1,t_2$.
If we already have a function $f(t_1,t_2,z)$ it is much easier to check
(even by hand) whether it is equal to the Hilbert series $H_{GL_2}(F_d^{\delta},t_1,t_2,z)$.
It is known that $f(t_1,t_2,z)=H_{GL_2}(F_d^{\delta},t_1,t_2,z)$ if and only if
\[
H_{GL_2}(F_d,t_1,t_2,z)=\frac{1}{t_1-t_2}(t_1f(t_1,t_2,z)-t_2f(t_2,t_1,z)).
\]
But in our situation, given $H_{GL_2}(F_d,t_1,t_2,z)$ only, we want to compute $H_{GL_2}(F_d^{\delta},t_1,t_2,z)$.
For this purpose we use the Elliott-McMahon method.
The Hilbert series $H_{GL_2}(F_d,t_1,t_2,z)$ is a so called nice rational symmetric function in $t_1,t_2$, i.e.,
its denominator is a product of binomials of the form $1-t_1^{m_1}t_2^{m_2}z^m$.
The Schur function $S_{(\lambda_1,\lambda_2)}(t_1,t_2)$ has the expression
\[
S_{(\lambda_1,\lambda_2)}(t_1,t_2)=(t_1t_2)^{\lambda_2}\frac{t_1^{\lambda_1-\lambda_2+1}-t_2^{\lambda_1-\lambda_2+1}}{t_1-t_2}.
\]
If
\[
H_{GL_2}(F_d,t_1,t_2,z)=\sum_{n\geq 0}\sum_{(\lambda_1,\lambda_2)}m(\lambda_1,\lambda_2,n)S_{(\lambda_1,\lambda_2)}(t_1,t_2)z^n
=\sum_{n\geq 0}h_n(t_1,t_2)z^n,
\]
then
\[
(t_1-t_2)H_{GL_2}(F_d,t_1,t_2,z)=\sum_{n\geq 0}\sum_{(\lambda_1,\lambda_2)}m(\lambda_1,\lambda_2,n)
(t_1^{\lambda_1+1}t_2^{\lambda_2}-t_1^{\lambda_2}t_2^{\lambda_1+1})z^n
\]
\[
=\sum_{n\geq 0}\sum_{n_1,n_2}a_{n_1,n_2}t_1^{n_1}t_2^{n_2}z^n.
\]
Hence we have to calculate ``half'' of the series $(t_1-t_2)H_{GL_2}(F_d,t_1,t_2,z)$, i.e., to evaluate the component
with $n_1>n_2$. Then
\[
H_{GL_2}(F_d^{\delta},t_1,t_2,z)=\frac{1}{t_1}\sum_{n\geq 0}\sum_{n_1>n_2}a_{n_1,n_2}t_1^{n_1}t_2^{n_2}z^n
\]
\[
=\sum_{n\geq 0}\sum_{(\lambda_1,\lambda_2)}m(\lambda_1,\lambda_2,n)
t_1^{\lambda_1}t_2^{\lambda_2}z^n.
\]
The idea of Elliott and McMahon is to replace $t_1$ and $t_2$ with $t_1u$ and $t_2/u$, respectively, and to obtain the Laurent series
\[
\left(t_1u-\frac{t_2}{u}\right)H_{GL_2}(F_d,t_1u,\frac{t_2}{u},z)=\sum_{k=-\infty}^{+\infty}\left(
\sum_{n\geq 0}\sum_{n_1-n_2=k)}a_{n_1,n_2}t_1^{n_1}t_2^{n_2}z^n\right)u^k
\]
\[
=\sum_{k=-\infty}^{+\infty}f_k(t_1,t_2,z)u^k=\sum_{k<0}f_k(t_1,t_2,z)u^k+\sum_{k>0}f_k(t_1,t_2,z)u^k.
\]
The Hilbert series we need is
\[
H_{GL_2}(F_d^{\delta},t_1,t_2,z)=\frac{1}{t_1}\sum_{k>0}f_k(t_1,t_2,z).
\]
For a nice rational function
\[
f(t_1,t_2,z)=\sum_{n\geq 0}\sum_{n_1,n_2}a_{n_1,n_2}t_1^{n_1}t_2^{n_2}z^n=p(t_1,t_2,z)\prod\frac{1}{1-t_1^{m_1}t_2^{m_2}z^m},
\]
$p(t_1,t_2,z)\in K[t_1,t_2,z]$, we can use two ways to calculate the component with $n_1>n_2$.
In both cases we replace $t_1$ and $t_2$, respectively, with $t_1u$ and $t_2/u$ in
$f(t_1,t_2,z)$.
The first method uses the approach of Elliott \cite{E}. The denominator of
$f(t_1u,t_2/u,z)$ is a product of three kinds of binomials:
\[
1-t_1^{m_1}t_2^{m_2}z^mu^k,\quad 1-t_1^{n_1}t_2^{n_2}z^nu^{-l},\quad k,l>0,\quad 1-t_1^{r_1}t_2^{r_2}z^r.
\]
If both kinds $1-t_1^{m_1}t_2^{m_2}z^mu^k$ and $1-t_1^{n_1}t_2^{n_2}z^nu^{-l}$ appear, we use the identity
\[
\frac{1}{(1-v_1)(1-v_2)}=\frac{1}{1-v_1v_2}\left(\frac{1}{1-v_1}+\frac{1}{1-v_2}-1\right)
\]
for $v_1=t_1^{m_1}t_2^{m_2}z^mu^k$ and $v_2=t_1^{m_1}t_2^{m_2}z^mu^{-l}$ and present $f(t_1u,t_2/u,z)$ as a sum of nice rational functions
in which the product $(1-t_1^{m_1}t_2^{m_2}z^mu^k)(1-t_1^{n_1}t_2^{n_2}z^nu^{-l})$ in the denominator
is replaced by one of the three kinds of products
\[
(1-t_1^{s_1}t_2^{s_2}u^st^{k-l})(1-t_1^{m_1}t_2^{m_2}z^mu^k), (1-t_1^{s_1}t_2^{s_2}u^st^{k-l})(1-t_1^{n_1}t_2^{n_2}z^nu^{-l}),
1-t_1^{s_1}t_2^{s_2}u^st^{k-l}.
\]
Continuing this process, we present $f(t_1u,t_2/u,z)$ as a sum of a Laurent polynomial in $u$
and nice rational functions which do not contain both factors
$1-t_1^{m_1}t_2^{m_2}z^mu^k$ and $1-t_1^{n_1}t_2^{n_2}z^nu^{-l}$. Then the contribution to $\sum_{k>0}f_k(t_1,t_2,z)u^k$
is due to the monomials of positive degree with respect to $u$ in the Laurent polynomial,
the nice rational functions with denominators not depending on factors of the form $1-t_1^{n_1}t_2^{n_2}z^nu^{-l}$ and the constant terms of
the nice functions depending on $1-t_1^{n_1}t_2^{n_2}z^nu^{-l}$. In this way we also prove that the Hilbert series
$H_{GL_2(K)}(F_d^{\delta},t_1,t_2,z)$ is a nice rational function. We refer to \cite{BBD} for details.
In our computations we have used another method. Considering $f(t_1u,t_2/u,z)$ as a rational function in the variable $u$, we present it
as a sum of a polynomial $p_0(t_1,t_2,z,u)$ in $u$ with rational in $t_1,t_2,z$ coefficients and of partial fractions of the form
\[
\frac{p_1(t_1,t_2,z)}{u^k},\quad \frac{p_2(t_1,t_2,z)}{(1-t_1^{n_1}t_2^{n_2}z^nu^k)^r},\quad \frac{p_3(t_1,t_2,z)}{(1-t_1^{n_1}t_2^{n_2}z^nu^{-k})^r},
\]
where $p_i(t_1,t_2,z)$ are rational functions.
Then $H_{GL_2(K)}(F_d^{\delta},t_1,t_2,z)$ is a sum of $p_0(t_1,t_2,z,u)$, the fractions $p_2(t_1,t_2,z)/(1-t_1^{n_1}t_2^{n_2}z^nu^k)^r$
and the functions $p_2(t_1,t_2,z)$, after replacing $u$ with 1.

Now we shall give the Hilbert series of the subalgebras of constants of Weitzenb\"ock derivations
of free metabelian associative algebras with small number of generators. In some of the cases
we give both Hilbert series, as graded and bigraded vector spaces, because we shall use
the results in the last section of our paper. We do not give results for derivations with a one-dimensional
Jordan cell because we shall handle them in the next section. We shall work out in detail the case $d=3$ and $\delta=\delta(2)$ only.
The computations in the other cases are similar.

\begin{example}\label{details for Hilbert series for d=3}
Let $d=3$ and $\delta=\delta(2)$, i.e., the matrix of $\delta$ (acting on $KX_3$) consists of one Jordan $3\times 3$ block.
The Hilbert series of the commutator ideal $F_3'$ of $F_3$ with respect to the canonical $GL_3(K)$-action is
\[
H(F_3',z_1,z_2,z_3)=\frac{1}{(1-z_1)(1-z_2)(1-z_3)}
\]
\[
\times\left(1+(z_1+z_2+z_3-1)\frac{1}{(1-z_1)(1-z_2)(1-z_3)}\right).
\]
With respect to the action of $D_2(K)$ the elements $x_1,x_2,x_3$ of the basis of $KX_3$ are homogeneous of degree
$(2,0),(1,1), (0,2)$, respectively. Hence we replace the variables $z_1,z_2,z_3$ with $t_1^2z,t_1t_2z,t_2^2z$, respectively.
In this way the Hilbert series of the $GL_2(K)$-module $F_3'$ is
\[
H_{GL_2(K)}(F_3',t_1,t_2,z)=\frac{1}{(1-t_1^2z)(1-t_1t_2z)(1-t_2^2z)}
\]
\[
\times\left(1+((t_1^2+t_1t_2+t_2^2)z-1)\frac{1}{(1-t_1^2z)(1-t_1t_2z)(1-t_2^2z)}\right).
\]
Now we consider the function
\[
f(t_1,t_2,z)=(t_1-t_2)H_{GL_2(K)}(F_3',t_1,t_2,z),
\]
replace there $t_1$ with $t_1u$ and $t_2$ by $t_2/u$ and express $f(t_1u,t_2/u,z)$ as a sum of partial fractions with respect to $u$:
\[
\frac{t_1^2t_2zu}{(1-t_1^2zu^2)^2(1-t_1t_2z)^2(1-t_1^2t_2^2z^2)}
+\frac{(t_1^2t_2^2z^2-t_1t_2z-1)t_1^2t_2zu}{(1-t_1^2zu^2)(1-t_1t_2z)^2(1+t_1t_2z)(1-t_1^2t_2^2z^2)}
\]
\[
+\frac{t_1^3t_2^4z^3u}{(u^2-t_2^2z)(1-t_1t_2z)^2(1+t_1t_2z)(1-t_1^2t_2^2z^2)}
-\frac{t_1t_2^4z^2u}{(u^2-t_2^2z)^2(1-t_1t_2z)^2(1-t_1^2t_2^2z^2)}.
\]
The expansion as a Laurent series in $u$
\[
\frac{u}{u^2-t_2^2z}=\frac{1}{u(1-t_2^2z/u)}=\sum_{k\geq 1}\frac{(t_2^2z)^{k-1}}{u^k}
\]
contains only negative powers of $u$. Hence we have to remove the third summand in $f(t_1u,t_2/u,z)$, and similarly for the fourth summand.
Hence the Hilbert series $H_{GL_2(K)}(F_3',t_1,t_2,z)$ is obtained replacing $u$ by 1 in the sum
\[
\frac{t_1^2t_2zu}{(1-t_1^2zu^2)^2(1-t_1t_2z)^2(1-t_1^2t_2^2z^2)}
+\frac{(t_1^2t_2^2z^2-t_1t_2z-1)t_1^2t_2zu}{(1-t_1^2zu^2)(1-t_1t_2z)^2(1+t_1t_2z)(1-t_1^2t_2^2z^2)}.
\]
After simple computations we obtain
\[
H_{GL_2(K)}((F_3')^{\delta},t_1,t_2,z)=\frac{t_1^3t_2z^2(1+(t_1+t_2)t_2z-t_1^2t_2^2z^2)}{(1-t_1^2z)^2(1-t_1t_2z)^3(1+t_1t_2z)^2},
\]
\[
H((F_3')^{\delta},z)=H_{GL_2(K)}(F_3',1,1,z)=\frac{z^2(1+2z-z^2)}{(1-z)^3(1-z^2)^2},
\]
Similarly, starting from the Hilbert series of $K[X_3]$
\[
H(K[X_3],z_1,z_2,z_3)=\frac{1}{(1-z_1)(1-z_2)(1-z_3)}
\]
we obtain
\[
H_{GL_2}(K[X_3]^{\delta},t_1,t_2,z)=\frac{1}{(1-t_1^2z)(1-t_1^2t_2^2z^2)},
\]
\[
H(K[X_3]^{\delta},z)=\frac{1}{(1-z)(1-z^2)},
\]
\[
H(F_3^{\delta},z)=H(K[X_3]^{\delta},z)+H((F_3')^{\delta},z).
\]
When $\delta$ acts on the linear component of $K[U_3,V_3]$, its matrix consists of two Jordan $3\times 3$ blocks.
The Hilbert series of $K[U_3,V_3]$ with respect to the canonical $GL_6(K)$-action is
\[
H(K[U_3,V_3],z_1,\ldots,z_6)=\prod_{i=1}^6\frac{1}{1-z_i}.
\]
With respect to the action of $D_2(K)$ the elements $\{u_1,v_1\},\{u_2,v_2\}, \{u_3,v_3\}$ are homogeneous of degree
$(2,0),(1,1), (0,2)$, respectively. Hence we replace the variables $z_1,z_4$ with $t_1^2z$, $z_2,z_5$ with $t_1t_2z$,
and $z_3,z_6$ with $t_2^2z$. Again, decomposing $(t_1u-t_2/u)H_{GL_2(K)}(K[U_3,V_3],t_1u,t_2/u,z)$ as a sum of partial fractions with respect to $u$
and taking only the fractions which contribute with positive degrees in the expansion as a Laurent series in $u$, we obtain
\[
H_{GL_2}(K[U_3,V_3]^{\delta},t_1,t_2,z)=\frac{1+t_1^3t_2z^2}{(1-t_1^2z)^2(1-t_1^2t_2^2z^2)^3},
\]
\[
H(K[U_3,V_3]^{\delta},z)=\frac{1+z^2}{(1-z)^2(1-z^2)^3}.
\]
Once found, it is easy to check that $H_{GL_2(K)}(F_3^{\delta},t_1,t_2,z)$ is really the Hilbert series of
$F_3^{\delta}$ as a bigraded vector space. It is sufficient to verify the equality
\[
H_{GL_2}(F_3,t_1,t_2,z)=\frac{1}{t_1-t_2}(t_1H_{GL_2}(F_3^{\delta},t_1,t_2,z)-t_2H_{GL_2}(F_3^{\delta},t_2,t_1,z))
\]
which can be seen by direct computations.
\end{example}

\begin{example}\label{Hilbert series for small d}
Let $\delta=\delta(p_1,\ldots,p_s)$ be the Weitzenb\"ock derivation acting on $KX_d$ with Jordan cells of size $p_1+1,\ldots,p_s+1$.
Let $\delta$ act on the vector spaces $KU_d$ and $KV_d$ in the same way as on $KX_d$.
Then the Hilbert series of the algebras of constants $F_d^{\delta}$ and $K[U_d,V_d]^{\delta}$
are:

\noindent $d=2$, $\delta=\delta(1)$:
\[
H_{GL_2}(F_2^{\delta},t_1,t_2,z)=\frac{1}{1-t_1z}+\frac{t_1t_2z^2}{(1-t_1z)^2(1-t_1t_2z^2)};
\]

\[
H(F_2^{\delta},z)=\frac{1}{1-z}+\frac{z^2}{(1-z)^2(1-z^2)};
\]

\[
H_{GL_2}(K[U_2,V_2]^{\delta},t_1,t_2,z)=\frac{1}{(1-t_1z)^2(1-t_1t_2z^2)};
\]

\[
H(K[U_2,V_2]^{\delta},z)=\frac{1}{(1-z)^2(1-z^2)};
\]

\noindent $d=3$, $\delta=\delta(2)$:
\[
H_{GL_2}(F_3^{\delta},t_1,t_2,z)=\frac{1}{(1-t_1^2z)(1-t_1^2t_2^2z^2)}
+\frac{t_1^3t_2(1+t_1t_2z)(1+t_1t_2z+t_2^2z-t_1^2t_2^2z^2)z^2}{(1-t_1^2z)^2(1-t_1^2t_2^2z^2)^3};
\]

\[
H(F_3^{\delta},z)=\frac{1}{(1-z)(1-z^2)}+\frac{z^2(1+z)(1+2z-z^2)}{(1-z)^2(1-z^2)^3};
\]

\[
H_{GL_2}(K[U_3,V_3]^{\delta},t_1,t_2,z)=\frac{1+t_1^3t_2z^2}{(1-t_1^2z)^2(1-t_1^2t_2^2z^2)^3};
\]

\[
H(K[U_3,V_3]^{\delta},z)=\frac{1+z^2}{(1-z)^2(1-z^2)^3};
\]

\noindent $d=4$, $\delta=\delta(3)$:
\[
H(F_4^{\delta},z)=\frac{1-z+z^2}{(1-z)^2(1-z^4)}+\frac{z^2p(z)}{(1-z)^4(1-z^2)^2(1-z^4)^3},
\]
\[
p(z)=2+z+3z^2+4z^3-6z^4-13z^5+13z^6-14z^7+2z^8+9z^9-5z^{10}+4z^{11}+2z^{12};
\]

\[
H(K[U_4,V_4]^{\delta},z)=\frac{1-2z+4z^2-3z^4-3z^8+4z^{10}-2z^{11}+z^{12}}{(1-z)^4(1-z^2)^2(1-z^4)^3};
\]

\noindent $d=4$, $\delta=\delta(1,1)$:
\[
H(F_4^{\delta},z)=\frac{1}{(1-z)^2(1-z^2)}+\frac{z^2(4+2z+z^2-22z^3+9z^4+10z^5-3z^6-2z^7+z^8)}{(1-z)^4(1-z^2)^5},
\]

\[
H(K[U_4,V_4]^{\delta},z)=\frac{1+z^2-4z^3+z^4+z^6}{(1-z)^4(1-z^2)^5};
\]

\noindent $d=5$, $\delta=\delta(4)$:
\[
H(F_5^{\delta},z)=\frac{1-z+z^2}{(1-z)^2(1-z^2)(1-z^3)}+\frac{z^2q(z)}{(1-z)^5(1-z^2)^3(1-z^3)^3},
\]
\[
q(z)=2+4z+5z^2-6z^3-15z^4+11z^5-10z^6+3z^7+5z^8+2z^9+z^{10}-5z^{11}+4z^{12}-z^{13};
\]

\[
H(K[U_5,V_5]^{\delta},z)=\frac{1-3z+6z^2-7z^3+3z^4+2z^5+z^6-9z^7+8z^8-3z^9+z^{10}}{(1-z)^5(1-z^2)^3(1-z^3)^3};
\]

\noindent $d=5$, $\delta=\delta(2,1)$:
\[
H(F_5^{\delta},z)=\frac{1+z^2}{(1-z)^2(1-z^2)(1-z^3)}+\frac{z^2r(z)}{(1-z)^5(1-z^2)^3(1-z^3)^3},
\]
\[
r(z)=4+8z+8z^2-9z^3-20z^4-5z^5-2z^6+9z^7+7z^8-2z^{11}+3z^{12}-z^{13};
\]

\[
H(K[U_5,V_5]^{\delta},z)=\frac{1+5z^2+z^3-4z^4-3z^5-3z^6-4z^7+z^8+5z^9+z^{11}}{(1-z)^5(1-z^2)^3(1-z^3)^3}.
\]
As in Example \ref{details for Hilbert series for d=3} we can check these results using the equality
\[
H_{GL_2}(F_d,t_1,t_2,z)=\frac{1}{t_1-t_2}(t_1H_{GL_2}(F_d^{\delta},t_1,t_2,z)-t_2H_{GL_2}(F_d^{\delta},t_2,t_1,z)).
\]
\end{example}

\section{Derivations with one-dimensional Jordan cell}

In this section we study the Weitzenb\"ock derivations $\delta$ of the form $\delta=\delta(p_1,\ldots,p_{s-1},0)$.
In other words, the Jordan normal form $J(\delta)$ of the linear operator $\delta$ acting on $KX_d$ has a $1\times 1$ cell.
Without loss of generality we may assume that $\delta$ acts on the vector spaces $KX_d$, $KU_d$, and $KV_d$
as a nilpotent linear operator such that
the vector subspaces  $KX_{d-1}$, $KU_{d-1}$, and $KV_{d-1}$ are $\delta$-invariant, and $\delta(x_d)=\delta(u_d)=\delta(v_d)=0$.
Now let $\omega(K[V_{d-1}])$ denote the augmentation ideal of $K[V_{d-1}]$, i.e., the ideal consisting all polynomials without constant term,
and let $K[U_{d-1},V_{d-1}]_{\omega}$ denote the tensor product of the $K$-algebras $K[U_{d-1}]$ and $\omega(K[V_{d-1}])$, i.e.,
\[
K[U_{d-1},V_{d-1}]_{\omega}=K[U_{d-1}]\otimes_K \omega(K[V_{d-1}]).
\]
This algebra can be considered as an ideal of the algebra $K[U_{d-1},V_{d-1}]$, and as a $K[U_{d-1},V_{d-1}]$-module generated by $V_{d-1}$.
By \cite[Proposition 3]{D2}
the algebra of constants $K[U_{d-1},V_{d-1}]_{\omega}^{\delta}$ of $K[U_{d-1},V_{d-1}]_{\omega}$
is a finitely generated $K[U_{d-1},V_{d-1}]^{\delta}$-module.

\begin{lemma}\label{reduction of Hilbert series}
The Hilbert series of $(F_d')^{\delta}$, $(F_{d-1}')^{\delta}$ and $K[U_{d-1},V_{d-1}]_{\omega}^{\delta}$ are related by
\[
H_{GL_2}((F_d')^{\delta},t_1,t_2,z)=\frac{1}{(1-z)^2}\left(H_{GL_2}((F_{d-1}')^{\delta},t_1,t_2,z)\right.
\]
\[
\left. +zH_{GL_2}(K[U_{d-1},V_{d-1}]_{\omega}^{\delta},t_1,t_2,z)\right),
\]
\[
H((F_d')^{\delta},z)=\frac{1}{(1-z)^2}(H((F_{d-1}')^{\delta},z)+zH(K[U_{d-1},V_{d-1}]_{\omega}^{\delta},z)).
\]
\end{lemma}

\begin{proof}
If $\delta$ acts on $KU_{d-1}$ and $KV_{d-1}$ as $\delta=\delta(p_1,\ldots,p_{s-1})$, then
it acts on $KU_d$ and $KV_d$ as $\delta(p_1,\ldots,p_{s-1},0)$.
By Lemma \ref{Hilbert series of free metabelian associative algebra} the Hilbert series of the commutator ideal of $F_d$ is
\[
H(F_d',z_1,\ldots,z_d)=\prod_{j=1}^d\frac{1}{1-z_j}H(L_d'/L_d'',z_1,\ldots,z_d),
\]
where $L_d'/L_d''$ is the commutator ideal of the free metabelian Lie algebra.
Following the procedure described in Section \ref{section Hilbert series} we replace its variables $z_j$ with
$t_1^{q_j}t_2^{r_j}z$, where the nonnegative integers $q_j,r_j$ depend on the size of the corresponding Jordan cell
and the position of the variable $x_j$ in the Jordan basis of $KX_d$. In particular, we have to replace the variable $z_d$ with $z$.
Hence
\[
H_{GL_2}(F_d',t_1,t_2,z)=\frac{1}{1-z}\prod_{j=1}^{d-1}\frac{1}{1-t_1^{q_j}t_2^{r_j}z}H_{GL_2}(L_d'/L_d'',t_1,t_2,z).
\]
We know from \cite[Lemma 4.1]{DDF} that
\[
H_{GL_2}(L_d'/L_d'',t_1,t_2,z)=\frac{1}{1-z}\left(H_{GL_2}(L_{d-1}'/L_{d-1}'',t_1,t_2,z)\right.
\]
\[
\left. +zH_{GL_2}(\omega(K[V_{d-1}]),t_1,t_2,z)\right).
\]
Thus
\[
H_{GL_2}(F_d',t_1,t_2,z)=\frac{1}{(1-z)^2}\left(H_{GL_2}(F_{d-1}',t_1,t_2,z)\right.
\]
\[
\left. +z(H_{GL_2}(K[U_{d-1}],t_1,t_2,z)H_{GL_2}(\omega(K[V_{d-1}]),t_1,t_2,z))\right)
\]
\[
=\frac{1}{(1-z)^2}(H_{GL_2}(F_{d-1}',t_1,t_2,z)+zH_{GL_2}(K[U_{d-1},V_{d-1}]_{\omega},t_1,t_2,z)).
\]
Using the fact that the formal power series $f(t_1,t_2,z)$ is the Hilbert series of the $\delta$-constants of
the $\mathbb Z$-graded $GL_2(K)$-module $W$ if and only if
\[
H_{GL_2}(W,t_1,t_2,z)=\frac{1}{t_1-t_2}(t_1f(t_1,t_2,z)-t_2f(t_2,t_1,z)),
\]
we obtain that the equality
\[
H_{GL_2}(F_d')=\frac{1}{(1-z)^2}(H_{GL_2}(F_{d-1}')+zH_{GL_2}(K[U_{d-1},V_{d-1}]_{\omega}))
\]
is equivalent to the desired equality
\[
H_{GL_2}((F_d')^{\delta},t_1,t_2,z)=\frac{1}{(1-z)^2}\left(H_{GL_2}((F_{d-1}')^{\delta},t_1,t_2,z)\right.
\]
\[
\left. +zH_{GL_2}(K[U_{d-1},V_{d-1}]_{\omega}^{\delta},t_1,t_2,z)\right).
\]
Clearly this implies that
\[
H((F_d')^{\delta},z)=\frac{1}{(1-z)^2}(H((F_{d-1}')^{\delta},z)+zH(K[U_{d-1},V_{d-1}]_{\omega}^{\delta},z)).
\]
\end{proof}

Recall that $F_d'$ is a $K[U_d,V_d]$-module.
For every monomial
\[
u_{j_1}\cdots u_{j_m}\in K[U_{d-1}], v_{j_1}\cdots v_{j_n}\in \omega(K[V_{d-1}]),m, n\geq 1,
\]
we define a $K$-linear map $\pi:K[U_{d-1},V_{d-1}]_{\omega}\to F_d'$ by
\[
\pi(v_{j_1}\cdots v_{j_n})=\sum_{k=1}^n[x_d,x_{j_k}]v_{j_1}\cdots v_{j_{k-1}}v_{j_{k+1}}\cdots v_{j_n},
\]
\[
\pi(u_{i_1}\cdots u_{i_m}v_{j_1}\cdots v_{j_n})=\pi(v_{j_1}\cdots v_{j_n})u_{i_1}\cdots u_{i_m},
\]
and
\[
\pi(u_{i_1}\cdots u_{i_m})=0.
\]

\begin{lemma}\label{pi commutes with delta}
{\rm (i) (\cite {DDF})} The map $\pi$ satisfies the equality
\[
\pi(v_1v_2)=\pi(v_1)v_2+\pi(v_2)v_1,\quad v_1,v_2\in \omega(K[V_{d-1}]).
\]
{\rm (ii)} The derivation $\delta$ and the map $\pi$ commute.
\end{lemma}

\begin{proof}
The case (i) was already proved in \cite{DDF}. It is sufficient to prove the case (ii) when
$w=uv\in K[U_{d-1},V_{d-1}]_{\omega}$,
and $u\in K[U_{d-1}]$ and $v\in \omega(K[V_{d-1}])$ are monomials. The map $\pi$ sends $v$ to
the commutator ideal $L_d'/L_d''$ of the free metabelian Lie algebra $L_d'/L_d''$. In \cite{DDF} we have seen that
$\delta(\pi(v))=\pi(\delta(v))$. Now
\[
\delta(\pi(w))=\delta(\pi(uv))=\delta(\pi(v)u)=\delta(\pi(v))u+\pi(v)\delta(u)=\pi(\delta(v))u+\pi(v)\delta(u)
\]
\[
=\pi(\delta(v)u)+\pi(v\delta(u))
=\pi(\delta(v)u+v\delta(u))=\pi(\delta(vu))=\pi(\delta(w))
\]
and this establishes (ii).
\end{proof}

The next theorem and its corollary are the main results of the section.

\begin{theorem}\label{reduction to d-1}
Let the Weitzenb\"ock derivation $\delta$ acting on the vector space $KX_d$
have a Jordan form with a $1\times 1$ cell. Let
$\delta$ act as a nilpotent linear operator
on $KX_{d-1}$ and $\delta(x_d)=0$, with a similar action on
$KU_d$ and $KV_d$.
Let $\{d_i\mid i\in I\}$, $\{g_j\in K[U_{d-1},V_{d-1}]_{\omega}\mid j\in J \}$
be, respectively, homogeneous vector space bases
of the $K[U_{d-1},V_{d-1}]^{\delta}$-module $(F_{d-1}')^{\delta}$ and of $K[U_{d-1},V_{d-1}]_{\omega}^{\delta}$
with respect to both $\mathbb Z$- and ${\mathbb Z}^2$-gradings.
Then $(F_d')^{\delta}$ has a basis
\[
\{d_iu_d^mv_d^n,\pi(g_j)u_d^mv_d^n \mid i\in I, j\in J,m,n\geq 0\}.
\]
\end{theorem}

\begin{proof}
The Hilbert series of $(F_{d-1}')^{\delta}$ and $K[U_{d-1},V_{d-1}]_{\omega}^{\delta}$ are equal, respectively,
to the generating functions of their bases. Hence
\[
H_{GL_2}((F_{d-1}')^{\delta},t_1,t_2,z)=\sum_{i\in I}t_1^{q_i}t_2^{r_i}z^{m_i},
\]
where $d_i$ is of bidegree $(q_i,r_i)$ and of total degree $m_i$. Since $u_d$ and $v_d$ are of bidegree $(0,0)$ and of total degree 1,
the generating function of the set $D=\{d_iu_d^mv_d^n\mid i\in I,n\geq 0\}$ is
\[
G(D,t_1,t_2,z)=\sum_{m,n\geq 0}\sum_{i\in I}t_1^{q_i}t_2^{r_i}z^{m_i}z^{m+n}=\frac{1}{(1-z)^2}H_{GL_2}((F_{d-1}')^{\delta},t_1,t_2,z).
\]
The map $\pi$ sends the monomials of $K[U_{d-1},V_{d-1}]_{\omega}^{\delta}$ to linear combinations of commutators with an extra variable $x_d$
in the beginning of each commutator. Hence, if the Hilbert series of $K[U_{d-1},V_{d-1}]_{\omega}^{\delta}$ is
\[
H_{GL_2}(K[U_{d-1},V_{d-1}]_{\omega}^{\delta},t_1,t_2,z)=\sum_{j\in J}t_1^{k_j}t_2^{l_j}z^{n_j},
\]
where the bidegree of $g_j$ is $(k_j,l_j)$ and its total degree is $n_j$, then
the generating function of the set $E=\{\pi(g_j)u_d^{m'}v_d^{n'}\mid j\in J,m',n'\geq 0\}$ is
\[
G(E,t_1,t_2,z)=\sum_{m,n\geq 0}\sum_{j\in J}t_1^{k_j}t_2^{l_j}z^{n_j+1}z^{m+n}
=\frac{z}{(1-z)^2}H_{GL_2}(K[U_{d-1},V_{d-1}]_{\omega}^{\delta},t_1,t_2,z).
\]
Hence, by Lemma \ref{reduction of Hilbert series}
\[
H_{GL_2}((F_d')^{\delta},t_1,t_2,z)=G(D,t_1,t_2,z)+G(E,t_1,t_2,z).
\]
Since both sets $D$ and $E$ are contained in $(F_d')^{\delta}$, we shall conclude that
$D\cup E$ is a basis of $(F_d')^{\delta}$ if we show that the elements of $D\cup E$
are linearly independent. For this purpose it is more convenient to work in the wreath product
$W_d$.

The elements $d_i$ belong to $F_{d-1}'\subset W_{d-1}$ and hence are of the form
\[
d_i=\sum_{k=1}^{d-1}a_kw_{ki}(U_{d-1},V_{d-1}),\quad w_{ki}(U_{d-1},V_{d-1})\in K[U_{d-1},V_{d-1}]_{\omega}.
\]
Therefore
\[
d_iu_d^mv_d^n=\sum_{k=1}^{d-1}a_kw_{ki}(U_{d-1},V_{d-1})u_d^mv_d^n.
\]
On the other hand, the elements $\pi(g_j)u_d^mv_d^n$ are of the form
\[
\pi(g_j)u_d^{m'}v_d^{n'}=\sum_{k=1}^{d-1}[x_d,x_k]h_{kj}(U_{d-1},V_{d-1})u_d^{m'}v_d^{n'}
\]
\[
=\sum_{k=1}^{d-1}(a_dv_k-a_kv_d)h_{kj}(U_{d-1},V_{d-1})u_d^{m'}v_d^{n'},
\]
where $h_{kj}(U_{d-1},V_{d-1})\in K[U_{d-1},V_{d-1}]$. Let
\[
w=\sum\xi_{imn}d_iu_d^mv_d^n+\sum\eta_{jm'n'}\pi(g_j)u_d^{m'}v_d^{n'}=0
\]
for some $\xi_{imn},\eta_{jm'n'}\in K$.
Clearly, we may assume that the elements $d_iu_d^mv_d^n$, $\pi(g_j)u_d^{m'}v_d^{n'}$ are homogeneous
with respect to each of the groups of variables $U_{d-1}$, $A_{d-1}\cup V_{d-1}$, $\{u_d\}$, and $\{a_d,v_d\}$.
It follows from the definition of $\pi$ that
\[
\pi(v_{j_1}\cdots v_{j_q})=\sum_{k=1}^q[x_d,x_{j_k}]v_{j_1}\cdots v_{j_{k-1}}v_{j_{k+1}}\cdots v_{j_q}
\]
\[
=\sum_{k=1}^q(a_dv_{j_k}-a_{j_k}v_d)v_{j_1}\cdots v_{j_{k-1}}v_{j_{k+1}}\cdots v_{j_q}
\]
\[
=qa_dv_{j_1}\cdots v_{j_q}-\sum_{k=1}^qa_{j_k}v_{j_1}\cdots v_{j_{k-1}}v_{j_{k+1}}\cdots v_{j_q}v_d.
\]
Hence, if $g_j(U_{d-1},V_{d-1})$ is homogeneous of degree $q$ with respect to $V_{d-1}$, then
\[
\pi(g_j)u_d^{m'}v_d^{n'}=\left(qa_dg_j(U_{d-1},V_{d-1})-\sum_{k=1}^{d-1}a_kh_{kj}(U_{d-1},V_{d-1})v_d\right)u_d^{m'}v_d^{n'}.
\]
In the linear dependence between the elements $d_iu_d^mv_d^n$ and $\pi(g_j)u_d^{m'}v_d^{n'}$ we replace
these elements with their expressions in $W_d$ and obtain
\[
w=\sum\xi_{imn}\sum_{k=1}^{d-1}a_kw_{ki}(U_{d-1},V_{d-1})u_d^mv_d^n
\]
\[
+\sum\eta_{jm'n'}\left(qa_dg_j(U_{d-1},V_{d-1})-\sum_{k=1}^{d-1}a_kh_{kj}(U_{d-1},V_{d-1})v_d\right)u_d^{m'}v_d^{n'}=0.
\]
Hence $m'=m$, $n'=n-1$, and, canceling $u_d^m$ and $v_d^{n-1}$, we obtain
\[
w=qa_d\sum\eta_{j,m,n-1}g_j(U_{d-1},V_{d-1})-\sum\eta_{j,m,n-1}\sum_{k=1}^{d-1}a_kh_{kj}(U_{d-1},V_{d-1})v_d
\]
\[
+\sum\xi_{imn}\sum_{k=1}^{d-1}a_kw_{ki}(U_{d-1},V_{d-1})v_d=0.
\]
Since the elements $g_j$ are linearly independent in $K[U_{d-1},V_{d-1}]_{\omega}^{\delta}$, comparing the coefficient of
$a_d$ in $w$ we conclude that $\eta_{j,m,n-1}=0$. Then, using that the elements $d_i$ are linearly independent in
$(F_{d-1}')^{\delta}$, we derive that $\xi_{imn}=0$. Hence the set $D\cup E$ is a basis of $(F_d')^{\delta}$.
\end{proof}

Every polynomial $f(U_{d-1},V_{d-1})\in K[U_{d-1},V_{d-1}]^{\delta}$ can be written in the form
\[
f(U_{d-1},V_{d-1})=f'(U_{d-1})+f''(U_{d-1},V_{d-1}),
\]
where $f'(U_{d-1})$ does not depend on $V_{d-1}$ and every monomial of $f''(U_{d-1},V_{d-1})$ contains a variable from $V_{d-1}$.
Since $\delta(f'(U_{d-1}))$ and $\delta(f''(U_{d-1},V_{d-1}))$ preserve these properties,
both $f'(U_{d-1})$ and $f''(U_{d-1},V_{d-1})$ belong to $K[U_{d-1},V_{d-1}]^{\delta}$.
Hence we may fix a system of generators of the algebra $K[U_{d-1},V_{d-1}]^{\delta}$
\[
\{e_1(U_{d-1}),\ldots,e_k(U_{d-1}),f_1(U_{d-1},V_{d-1}),\ldots,f_l(U_{d-1},V_{d-1})\},
\]
where $e_i(U_{d-1})\in K[U_{d-1}]^{\delta}$ and all monomials of $f_j(U_{d-1},V_{d-1})$ depend on $V_{d-1}$, i.e.,
$f_j(U_{d-1},V_{d-1})$ belongs to $K[U_{d-1},V_{d-1}]_{\omega}^{\delta}$. Every element of $K[U_{d-1},V_{d-1}]_{\omega}^{\delta}$
is a linear combination of products $e_1^{a_1}\cdots e_k^{a_k}f_1^{b_1}\cdots f_l^{b_l}$, where $b_1+\cdots+b_l>0$.
Hence the set
\[
\{f_1(U_{d-1},V_{d-1}),\ldots,f_l(U_{d-1},V_{d-1})\}
\]
generates the $K[U_{d-1},V_{d-1}]^{\delta}$-module $K[U_{d-1},V_{d-1}]_{\omega}^{\delta}$.
We also fix a set $\{c_1,\ldots,c_m\}$ of generators of the  $K[U_{d-1},V_{d-1}]^{\delta}$-module
$(F_{d-1}')^{\delta}$.
Without loss of generality we may assume that the polynomials in these systems are homogeneous.
Our purpose is to find a generating set of the $K[U_d,V_d]^{\delta}$-module $(F_d')^{\delta}$.

\begin{corollary}\label{obtaining generators from smaller sets}
Let $X_d$ be a Jordan basis of the Weitzenb\"ock derivation $\delta$ and let $\delta$ have a $1\times 1$ Jordan cell
corresponding to $x_d$. Let
$\delta$ act in the same way on $KU_d$ and $KV_d$.
Let $\{c_1,\ldots,c_m\}$ be a homogeneous generating set
of the $K[U_{d-1},V_{d-1}]^{\delta}$-module $(F_{d-1}')^{\delta}$.
Let $\{e_1,\ldots,e_k,f_1,\ldots,f_l\}$ be a generating set of the algebra $K[U_{d-1},V_{d-1}]^{\delta}$,
where $e_1,\ldots,e_k$ do not depend on $V_{d-1}$ and every monomial of $f_1,\ldots,f_l$ depends on $V_{d-1}$.
Then the $K[U_d,V_d]^{\delta}$-module $(F_d')^{\delta}$ is generated by the set
\[
\{c_1,\ldots,c_m\}\cup\{\pi(f_1),\ldots,\pi(f_l)\}.
\]
\end{corollary}

\begin{proof}
Clearly, the $K[U_{d-1},V_{d-1}]^{\delta}$-module $(F_{d-1}')^{\delta}$ is spanned by the elements
$c_ie_1^{a_1}\cdots e_k^{a_k}f_1^{b_1}\cdots f_l^{b_l}$.
In particular, in this way we obtain all elements $d_i$ from the basis of
the vector space $(F_{d-1}')^{\delta}$.
Since $u_d,v_d\in K[U_d,V_d]^{\delta}$, we obtain also all elements $d_iu_d^mv_d^n$.
Similarly, the $K[U_{d-1},V_{d-1}]^{\delta}$-module $K[U_{d-1},V_{d-1}]_{\omega}^{\delta}$ is spanned by
the products $e_1^{a_1}\cdots e_k^{a_k}f_1^{b_1}\cdots f_l^{b_l}$ which satisfy the property $b_1+\cdots+b_l>0$.
By Lemma \ref{pi commutes with delta}
\[
\pi(e_1^{a_1}\cdots e_k^{a_k}f_1^{b_1}\cdots f_l^{b_l})
=\sum_{j=1}^l\pi(f_j)e_1^{a_1}\cdots e_k^{a_k}f_1^{b_1}\cdots f_j^{b_j-1}\cdots f_l^{b_l}.
\]
Since $K[U_d,V_d]^{\delta}=(K[U_{d-1},V_{d-1}]^{\delta})[u_d,v_d]$, the elements $\pi(g_j)u_d^mv_d^n$ belong
to the $K[U_d,V_d]^{\delta}$-module generated by $\pi(f_1),\ldots,\pi(f_l)$.
Hence
$\{c_1,\ldots,c_m\}\cup\{\pi(f_1),\ldots,\pi(f_l)\}$ generate
the $K[U_d,V_d]^{\delta}$-module $(F_d')^{\delta}$.
\end{proof}

\section{Generating sets for small number of generators}

In this section we shall find the generators of the $K[U_d,V_d]^{\delta}$-module $(F_d')^{\delta}$
for $d\leq 3$.

\begin{example}\label{block of 2}
Let $d=2$, $\delta=\delta(1)$, and let $\delta(x_1)=0$, $\delta(x_2)=x_1$.
It is well known, see e.g., \cite{N}, that
$K[U_2,V_2]^{\delta}$ is generated by the algebraically independent polynomials
$u_1$, $v_1$ and $u_1v_2-u_2 v_1$. Hence
\[
H_{GL_2}(K[U_2,V_2]^{\delta},t_1,t_2,z)=\frac{1}{(1-t_1z)^2(1-t_1t_2z^2)},
\]
as indicated in Example \ref{Hilbert series for small d}.
The same example gives that
\[
H_{GL_2}((F_2')^{\delta},t_1,t_2,z)=\frac{t_1t_2z^2}{(1-t_1z)^2(1-t_1t_2z^2)}.
\]
It is easy to see that the element $[x_2,x_1]$ is
belong to $(F_2')^{\delta}$ and is of bidegree $(1,1)$. Since the Hilbert series
of $K[U_2,V_2]^{\delta}$-submodule of $(F_2')^{\delta}$ generated by $[x_2,x_1]$ is equal to the Hilbert series
of the whole module $(F_2')^{\delta}$, we conclude that $(F_2')^{\delta}$ is a free cyclic $K[U_2,V_2]^{\delta}$-module generated by $[x_2,x_1]$.
As a vector space $(F_2)^{\delta}$ is spanned by the elements $x_1^k$, $[x_2,x_1]u_1^lv_1^m(u_1v_2-u_2v_1)^n$, $k,l,m,n \geq 0$.
Recall that the action of $u_1v_2-u_2v_1$ on the commutator ideal $F_2'$ is given by
\[
w(u_1v_2-u_2v_1)=x_1[w,x_2]-x_2[w,x_1],\quad w\in F_2'.
\]
Knowing the basis of $F_2^{\delta}$ it is easy to derive that
the algebra $(F_2)^{\delta}$ is generated by the infinite set
\[
\{x_1, [x_1,x_2](u_1v_2-u_2v_1)^n\mid n\geq 0\}.
\]
\end{example}

\begin{example}\label{blocks of 2 and 1}
Let $d=3$ and let the Jordan normal form of $\delta$ have two cells, of size $2\times 2$ and $1\times 1$, respectively.
Hence $\delta=\delta(1,0)$ in our notation and we may apply Corollary \ref{obtaining generators from smaller sets}.
By Example \ref{block of 2} for $d=2$ and $\delta=\delta(1)$,
the algebra $K[U_2,V_2]^{\delta}$ is generated by $u_1$, $v_1$ and $u_1v_2-u_2v_1$. The $K[U_2,V_2]^{\delta}$-module
$(F_2')^{\delta}$ is generated by the single element $[x_2,x_1]$. In the notation of Corollary \ref{obtaining generators from smaller sets},
\[
e_1=u_1,\quad f_1=v_1,\quad f_2=u_1v_2-u_2v_1,\quad c_1=[x_2,x_1].
\]
Hence the $K[U_3,V_3]^{\delta}$-module $(F_3')^{\delta}$ is generated by $c_1=[x_2,x_1]$, $\pi(f_1)=[x_3,x_1]$, and
\[
\pi(f_2)=\pi(u_1v_2-u_2v_1)=\pi(v_2)u_1-\pi(v_1)u_2=x_1[x_3,x_2]-x_2[x_3,x_1].
\]
These three elements satisfy the relation
\[
c_1u_1v_3-\pi(f_1)(u_1v_2-u_2v_1)+\pi(f_2)v_1=0,
\]
i.e.,
\[
x_1[[x_2,x_1],x_3]-(x_1[[x_3,x_1],x_2]-x_2[[x_3,x_1],x_1])+[x_1[x_3,x_2]-x_2[x_3,x_1],x_1]=0.
\]
This result agrees with the Hilbert series
\[
H_{GL_2}((F_3')^{\delta(1,0)},t_1,t_2,z)=\frac{t_1z^2+t_1t_2z^2+t_1t_2z^3-t_1^2t_2z^4}{(1-z)^2(1-t_1z)^2(1-t_1t_2z^2)}
\]
of $K[U_3,V_3]^{\delta}$-module $(F_3')^{\delta}$. (The summands $t_1z^2$, $t_1t_2z^2$, $t_1t_2z^3$ in the nominator
correspond to the three generators $c_1$, $\pi(f_1)$, $\pi(f_2)$, and $-t_1^2t_2z^4$ corresponds to the relation.)
\end{example}

\begin{example}\label{block of 3}
Let $d=3$, $\delta=\delta(2)$, and let $\delta(x_1)=0$, $\delta(x_2)=x_1$, $\delta(x_3)=x_2$.
The generators of $K[U_3,V_3]^{\delta}$ are given in \cite{N}. In our notation they are
\[
f_1=u_1,\quad f_2=v_1,\quad f_3=u_2^2-2u_1u_3,\quad f_4=v_2^2-2v_1v_3,
\]
\[
f_5=u_1v_3-u_2v_2+u_3v_1,\quad f_6=u_1v_2-u_2v_1.
\]
There are two more generators in \cite{N},
\[
f_7=2u_1^2v_3-2u_1u_2v_2+u_2^2v_1,\quad f_8=u_1v_2^2-2v_1u_2v_2+2u_3v_1^2,
\]
but they can be expressed by the other ones:
\[
f_7=f_2f_3+2f_1f_5,\quad f_8=f_1f_4+2f_2f_5.
\]
The generators $f_1,f_2,f_3,f_4,f_5,f_6$ of $K[U_3,V_3]^{\delta}$ satisfy the defining relation
\[
f_6^2=f_1^2f_4+f_2^2f_3+2f_1f_2f_5.
\]
We easily conclude from the Hilbert series
\[
H_{GL_2}(K[U_3,V_3]^{\delta},t_1,t_2,z)=\frac{1+t_1^3t_2z^2}{(1-t_1^2z)^2(1-t_1^2t_2^2z^2)^3}
\]
that $K[U_3,V_3]^{\delta}$ is a free $K[f_1,f_2,f_3,f_4,f_5]$-module with generators 1 and $f_6$, and that
the algebra $K[U_3,V_3]^{\delta}$ has the presentation
\[
K[U_3,V_3]^{\delta}\cong K[f_1,f_2,f_3,f_4,f_5,f_6\mid f_6^2=f_1^2f_4+f_2^2f_3+2f_1f_2f_5].
\]
In particular, as a vector space $K[U_3,V_3]^{\delta}$ has a basis
\[
\{f_1^{q_1}f_2^{q_2}f_3^{q_3}f_4^{q_4}f_5^{q_5},f_1^{q_1}f_2^{q_2}f_3^{q_3}f_4^{q_4}f_5^{q_5}f_6\mid q_1,q_2,q_3,q_4,q_5\geq 0\}.
\]
By Example \ref{Hilbert series for small d} the Hilbert series of $(F_3')^{\delta}$ is
\[
H_{GL_2}((F_3')^{\delta},t_1,t_2,z)=\frac{(1+t_1t_2z+t_2^2z-t_1^2t_2^2z^2)(t_1^3t_2z^2+t_1^4t_2^2z^3)}{(1-t_1^2z)^2(1-t_1^2t_2^2z^2)^3}.
\]
\[
=t_1^3t_2z^2(1+2t_1^2z)+t_1^2t_2^2z^3(2t_1^2+t_1t_2)+\cdots
\]
This suggests that the $K[U_3,V_3]^{\delta}$-module $(F_3')^{\delta}$ has a generator $c_1$ of degree $2$ and bidegree $(3,1)$. It
together with $c_1f_1$ and $c_1f_2$ give the contribution $t_1^3t_2z^2(1+2t_1^2z)$. We also expect three generators
$c_2$, $c_3$ and $c_4$ of degree $3$ and bidegree $(4,2)$, $(4,2)$ and $(3,3)$, respectively.
By easy calculations we have found the explicit form of $c_1,c_2,c_3,c_4$:
\[
c_1=[x_2,x_1],\quad c_2=[x_3,x_1]v_1-[x_2,x_1]v_2=[x_3,x_1,x_1]-[x_2,x_1,x_2],
\]
\[
c_3=[x_3,x_1]u_1-[x_2,x_1]u_2=x_1[x_3,x_1]-x_2[x_2,x_1],
\]
\[
c_4=[x_3,x_2]u_1-[x_3,x_1]u_2+[x_2,x_1]u_3=x_1[x_3,x_2]-x_2[x_3,x_1]+x_3[x_2,x_1].
\]
For example, $c_4$ is a linear combination of all elements of degree 3 and bidegree $(3,3)$ of the form:
$[x_3,x_2]u_1$, $[x_3,x_1]u_2$, and $[x_2,x_1]u_3$:
\[
c_4=\gamma_1x_1[x_3,x_2]+\gamma_2x_2[x_3,x_1]+\gamma_3x_3[x_2,x_1],
\quad \gamma_1,\gamma_2,\gamma_3\in K,
\]
and the condition $\delta(c_4)=0$ gives
\[
0=\gamma_1x_1[x_3,x_1]+\gamma_2(x_1[x_3,x_1]+x_2[x_2,x_1])+\gamma_3x_2[x_2,x_1]
\]
\[
=(\gamma_1+\gamma_2)x_1[x_3,x_1]+(\gamma_2+\gamma_3)x_2[x_2,x_1].
\]
Hence
\[
\gamma_1+\gamma_2=\gamma_2+\gamma_3=0
\]
and, up to a multiplicative constant, the only solution is
\[
\gamma_1=1,\quad \gamma_2=-1,\quad \gamma_3=1.
\]
The Hilbert series of the free $K[U_3,V_3]^{\delta}$-module generated by four elements of bidegree
$(3,1)$, $(4,2)$, $(4,2)$, and $(3,3)$ is
\[
H_{GL_2}(t_1,t_2,z)=\frac{t_1^3t_2z^2(1+(2t_1+t_2)t_2z)(1+t_1^3t_2z^2)}{(1-t_1^2z)^2(1-t_1^2t_2^2z^2)^3}.
\]
Hence
\[
H_{GL_2}(t_1,t_2,z)-H_{GL_2}((F_3')^{\delta},t_1,t_2,z)=(t_1^2-t_2^2)t_1^4t_2^2z^4+\cdots
\]
which suggests that there is a relation of bidegree $(6,2)$ and a generator of bidegree $(4,4)$.
Continuing in the same way, we have found one more generator
\[
c_5=[x_3,x_1]u_3v_1-[x_3,x_1]u_1v_3+[x_3,x_2]u_1v_2-[x_3,x_2]u_2v_1-[x_2,x_1]u_3v_2+[x_2,x_1]u_2v_3
\]
of bidegree $(4,4)$ and the relations
\[
R_1(6,2): c_1f_6=c_3f_2-c_2f_1,
\]
\[
R_2(7,3): c_2f_6=c_4f_2^2-c_1(f_1f_4+f_2f_5),
\]
\[
R_3(7,3): c_3f_6=c_4f_1f_2+c_1(f_1f_5+f_2f_3),
\]
\[
R_4(6,4): c_4f_6=c_2f_3+c_3f_5+c_5f_1,
\]
\[
R_5(6,4): c_5f_2=c_2f_5+c_3f_4,
\]
\[
R_6(7,5): c_5f_6=c_1(f_3f_4-f_5^2)+c_4(f_1f_4+f_2f_5),
\]
The above relations show that $c_jf_6$, $j=1,\ldots,5$, and $c_5f_2$ can be replaced with a linear combination of other generators.
Hence the $K[U_3,V_3]^{\delta}$-module generated by $c_1,\ldots,c_5$ is spanned by
\[
E=\{c_jf_1^{m_j}f_2^{n_j}f_3^{p_j}f_4^{q_j}f_5^{r_j}\mid m_j,n_j,p_j,q_j,r_j\geq 0,j=1,2,3,4\}
\]
\[
\cup \{c_5f_1^mf_3^pf_4^qf_5^r\mid m,p,q,r\geq 0\}.
\]
It is easy to check that the generating function of the set $E$ is equal to the Hilbert series of $(F_3')^{\delta}$.
Hence, if we show that the elements of $E$ are linearly independent, we shall conclude that
the $K[U_3,V_3]^{\delta}$-module $(F_3')^{\delta}$ is generated by $c_1,\ldots,c_5$.
Let
\[
\sum_{j=1}^5c_js_j=0,
\]
where $s_j$ are polynomials in $f_1,f_2,f_3,f_4,f_5$, $j=1,\ldots,5$, and $s_5$ does not depend on $f_2=v_1$.
We shall show that this implies that $s_j=0$, $j=1,\ldots,5$.
The bidegrees of the polynomials $f_1,f_2,f_3,f_4,f_5$ consist of pairs of even numbers
which implies that we can rewrite the equation above as
\[
c_1s_1+c_4s_4=0,\quad c_2s_2+c_3s_3+c_5s_5=0,
\]
since the only bidegrees consisting of odd numbers are the bidegrees of $c_1$ and $c_4$, which are $(3,1)$ and $(3,3)$, respectively.
We shall work in the abelian wreath product $W_3$ and shall denote by $w_i$ the coordinate of $a_i$
of $w\in W_3$. First we consider the equation $c_1s_1+c_4s_4=0$. The three coordinates $w_i$ of
\[
w=c_1s_1+c_4s_4=a_1w_1+a_2w_2+a_3w_3=0
\]
define a linear homogeneous system
\[
w_i=0,\quad i=1,2,3,
\]
with unknowns $s_1$ and $s_4$. Since $w_3=(u_1v_2-u_2v_1)s_4=0$, then $s_4=0$ and thus $s_1=0$. Now let us consider
\[
w=c_2s_2+c_3s_3+c_5s_5=0.
\]
We may assume that not all monomials of $s_2,s_3,s_5$ depend on $v_2$.
We substitute $v_2=0$ in the wreath product. Then $f_1,f_2,f_3,f_4,f_5$ become
\[
\bar f_1=u_1,\quad \bar f_2=v_1,\quad \bar f_3=u_2^2-2u_1u_3,
\]
\[
\bar f_4=-2v_1v_3, \quad \bar f_5=u_1v_3+u_3v_1
\]
and the generators $c_2,c_3$ and $c_5$ become
\[
\bar c_2=(-a_1v_3+a_3v_1)v_1, \quad \bar c_3=-a_1u_1v_3-a_2u_2v_1+a_3u_1v_1,
\]
\[
\bar c_5=a_1(u_1v_3-u_3v_1)v_3+a_2u_2v_1v_3-a_3(u_1v_3-u_3v_1)v_1
\]
in $\bar{W}_3$. Also, at least one of the polynomials $\bar s_2,\bar s_3,\bar s_5$ is different from 0.
The equality $\bar w=\bar c_2\bar s_2+\bar c_3\bar s_3+\bar c_5\bar s_5=0$ gives the equalities of the coordinates
\[
\bar w_1=-v_3(v_1\bar s_2+u_1\bar s_3-(u_1v_3-u_3v_1)\bar s_5)=0,
\]
\[
\bar w_2=-u_2v_1(\bar s_3-2v_3\bar s_5)=0,
\]
\[
v_1(v_1\bar s_2+u_1\bar s_3-(u_1v_3-u_3v_1)\bar s_5)=0.
\]
We consider these three equalities as a homogeneous system with unknowns $\bar s_2,\bar s_3,\bar s_5$. Its only solution is
\[
\bar s_2=\frac{(u_1v_3+u_3v_1)\bar s_5}{v_1},\quad \bar s_3=2v_3\bar s_5.
\]
Since we work with polynomials $\bar s_2,\bar s_3,\bar s_5$, we conclude that $v_1$ divides $\bar s_5$
which contradicts with the assumption that $s_5$ does not depend on $v_1$. Hence $\bar s_2=\bar s_3=\bar s_5=0$.
Since  $\bar f_1,\bar f_2,\bar f_3,\bar f_4$, and $\bar f_5$ are algebraically independent in $K[U_3,v_1,v_3]$,
we obtain that $s_2=s_3=s_5=0$. This completes the proof that the $K[U_3,V_3]^{\delta}$-module
$(F_3')^{\delta}$ is generated by $c_1,\ldots,c_5$.

Now we want to find generators of the algebra $F_3^{\delta}$. We start with $x_1$ and $x_2^2-2x_1x_3$
which generate $F_3^{\delta}$ modulo the ideal $(F_3')^{\delta}$.
We want to lift them to elements in $F_3^{\delta}$. Obviously, $x_1\in F_3^{\delta}$. It is easy to check that the
element $x_2^2-(x_1x_3+x_3x_1)$ belongs to $F_3^{\delta}$ and acts by multiplication on $F_3'$ in the same way as
$x_2^2-2x_1x_3$. Hence $x_2^2-(x_1x_3+x_3x_1)$ is the lifting of $x_2^2-2x_1x_3$ which we are searching for.
Then we shall take a subset of $E$
which, together with $x_1$ and $x_2^2-2x_1x_3$, generates $F_3^{\delta}$. If $w\in F_3'$, then
$wf_1=x_1w$, $wf_3=(x^2-(x_1x_3+x_3x_1))w$ and we can remove the elements of $E$
\[
c_jf_1^{m_j}f_2^{n_j}f_3^{p_j}f_4^{q_j}f_5^{r_j}, \quad j=1,2,3,4,\quad
c_5f_1^mf_3^pf_4^qf_5^r
\]
which contain $f_1$ and $f_3$. Hence
we may assume that $m_j,p_j,m,p=0$.
Also, $wf_2=wx_1-x_1w$ and we may consider only those elements of $E$ with $n_j=0$. Similarly,
\[
wf_4=w(x_2^2-(x_1x_3+x_3x_1))-(x_2^2-(x_1x_3+x_3x_1))w+2wf_5
\]
and we assume that $q_j=q=0$. Hence, as a vector space
\[
(F_3')^{\delta}=\sum_{j=1}^5K[x_1,x_2^2-2x_1x_3](c_jK[f_5])K[x_1,x_2^2-2x_1x_3].
\]
As a consequence we obtain that the algebra $F_3^{\delta}$ is generated by
\[
\{x_1,x_2^2-(x_1x_3+x_3x_1),c_jf_5^{p_j}\mid j=1,\ldots,5, p_j\geq 0\}.
\]
\end{example}

\section*{Acknowledgements}
The research of the first named author was a part of his project
in the frames of the High School Student Institute at the Institute of Mathematics and Informatics of the Bulgarian Academy of Sciences.
The research of the second named author was partially supported
by Grant Ukraine 01/0007 of the Bulgarian Science Fund for Bilateral Scientific Cooperation between Bulgaria and Ukraine.
The research of the third named author was partially supported by the
Council of Higher Education (Y\"OK) in Turkey during his visit
as a post-doctoral fellow at the Institute of Mathematics and Informatics of
the Bulgarian Academy of Sciences. He is very thankful to the Institute for the creative atmosphere and the warm hospitality
for the period when this project was carried out.
For the revised version of the paper,
the research of the second and the third named authors was partially supported by Grant I02/18 of the Bulgarian National Science Fund.

\end{document}